\documentclass[12pt,letterpaper,reqno]{amsart}

\usepackage{graphicx}
\usepackage{amsfonts}
\usepackage{amssymb}
\usepackage{amsthm}
\usepackage{amsmath}

\newcommand{\supp}{{\rm supp}}

\newcommand{\hphi}{\widehat{\phi}}  %phi^

\newcommand\be{\begin{equation}}
\newcommand\ee{\end{equation}}
\newcommand\bea{\begin{eqnarray}}
\newcommand\eea{\end{eqnarray}}
\newcommand\bi{\begin{itemize}}
\newcommand\ei{\end{itemize}}
\newcommand\ben{\begin{enumerate}}
\newcommand\een{\end{enumerate}}
\newcommand\bc{\begin{center}}
\newcommand\ec{\end{center}}
\newcommand\ba{\begin{array}}
\newcommand\ea{\end{array}}
\newcommand{\R}{\ensuremath{\mathbb{R}}}

\newcommand{\foh}{\frac{1}{2}}  %onehalf
\newtheorem{thm}{Theorem}[section]
\newtheorem{conj}[thm]{Conjecture}

\newtheorem{lem}[thm]{Lemma}

\theoremstyle{definition}
\newtheorem{rek}[thm]{Remark}

\newcommand{\twocase}[5]{#1 \begin{cases} #2 & \text{{\rm #3}}\\ #4
&\text{{\rm #5}} \end{cases}   }

\newcommand{\threecase}[7]{#1 \begin{cases} #2 & \text{{\rm #3}}\\ #4
&\text{{\rm #5}}\\ #6 & \text{{\rm #7}} \end{cases}   }

\newcommand{\gep}{\epsilon}

\newcommand{\gam}[2]{\Gamma\left(\frac{#1}{#2}\right)}

\numberwithin{equation}{section}

\begin{document}

\title{A Unitary Test of the Ratios Conjecture}

\author[Goes]{John Goes}\email{johnwgoes@gmail.com}
\address{Department of Mathematics, University of Illinois at Chicago, Chicago, IL 60607}

\author[Jackson]{Steven Jackson}\email{Steven.R.Jackson@williams.edu}
\address{Department of Mathematics and Statistics, Williams College, Williamstown, MA 01267}

\author[Miller]{Steven J. Miller}\email{Steven.J.Miller@williams.edu}
\address{Department of Mathematics and Statistics, Williams College, Williamstown, MA 01267}

\author[Montague]{David Montague}\email{davmont@umich.edu}
\address{Department of Mathematics, University of Michigan, Ann Arbor, MI 48109}

\author[Ninsuwan]{Kesinee Ninsuwan}\email{Kesinee\underline{\ }Ninsuwan@brown.edu}
\address{Department of Mathematics, Brown University, Providence, RI 02912}

\author[Peckner]{Ryan Peckner}\email{rpeckner@berkeley.edu}
\address{Department of Mathematics, University of California, Berkeley, CA 94720}

\author[Pham]{Thuy Pham}\email{tvp1@williams.edu}
\address{Department of Mathematics and Statistics, Williams College, Williamstown, MA 01267}

\subjclass[2010]{11M26 (primary), 11M41, 15B52 (secondary).}
\keywords{$1$-Level Density, Dirichlet $L$-functions, Low Lying
Zeros, Ratios Conjecture, Dirichlet Characters}

\date{\today}

\thanks{This work was done at the 2009 SMALL Undergraduate Research Project at Williams College, funded by NSF Grant DMS0850577 and Williams College; it is a pleasure to thank them and the other participants. We are grateful to Daniel Fiorilli for comments on an earlier draft, and the referee for many valuable suggestions and observations. The third named author was also partly supported by NSF Grant DMS0600848.}

%We thank Nina Snaith for helpful discussions on the Ratios Conjecture.

\begin{abstract} The Ratios Conjecture of Conrey, Farmer and Zirnbauer \cite{CFZ1,CFZ2} predicts the answers to numerous questions in number theory, ranging from $n$-level densities and correlations to mollifiers to moments and vanishing at the central point. The conjecture gives a recipe to generate these answers, which are believed to be correct up to square-root cancelation. These predictions have been verified, for suitably restricted test functions, for the 1-level density of orthogonal \cite{HuyMil,Mil5,MilMo} and symplectic \cite{Mil3,St} families of $L$-functions. In this paper we verify the conjecture's predictions for the unitary family of all Dirichlet $L$-functions with prime conductor; we show square-root agreement between prediction and number theory if the support of the Fourier transform of the test function is in $(-1,1)$, and for support up to $(-2,2)$ we show agreement up to a power savings in the family's cardinality.
\end{abstract}

\maketitle

\tableofcontents

%%%%%%%%%%%%%%%%%%%%%%%%%%%%%%%%%%%%%%%%%%%%%%%%%%%%%%%%%%%%%%%%%%%%%%%%%%%%%%%%%%%%%%%%%%%%%%%%
%%%%%%%%%%%%%%%%%%%%%%%%%%%%%%%%%%%%%%%%%%%%%%%%%%%%%%%%%%%%%%%%%%%%%%%%%%%%%%%%%%%%%%%%%%%%%%%%
%%%%%%%%%%%%%%%%%%%%%%%%%%%%%%%%%%%%%%%%%%%%%%%%%%%%%%%%%%%%%%%%%%%%%%%%%%%%%%%%%%%%%%%%%%%%%%%%

\section{Introduction}

As the solutions to many problems in number theory are governed by properties of $L$-functions, it is thus important to understand these objects. There are numerous examples of these connections, such as the relationship between the zeros of $\zeta(s)$ and the error term in the Prime Number Theorem (see for example \cite{Da,IK}), the Birch and Swinnerton-Dyer conjecture (which asserts that the rank of the Mordell-Weil group of rational solutions of an elliptic curve $E$ equals the order of vanishing of the associated $L$-function $L(s,E)$ at $s=1/2$; see for instance \cite{IK}), and the order of vanishing of $L$-functions at the central point or the number of normalized zeros of $L$-functions less than half the average spacing apart and the growth of the class number \cite{CI,Go,GZ}, to name just a few.

Since the 1970s, the zeros and values of $L$-functions have been successfully modeled by random matrix theory, which says that zeros behave like eigenvalues of random matrix ensembles, and values behave like the values of the corresponding characteristic polynomials. The correspondence was first seen in the work of Montgomery \cite{Mon2}. Early numerical support was provided by Odlyzko's investigations of the spacings between zeros of $L$-functions and eigenvalues of complex Hermitian matrices \cite{Od1,Od2}. Initially the most useful model for number theory was Dyson's circular unitary ensemble (CUE), a modification of the Gaussian unitary ensemble (GUE); later investigations of zeros near the central point showed that the scaling limits of the various classical compact groups (orthogonal, symplectic and unitary $N \times N$ matrices) are needed. For some of the history and summary of results, see \cite{Con,FM,KaSa2,KeSn3,Meh,MT-B,RS}. These models have led researchers to the correct answers to many problems, and, in fact, have suggested good questions to ask. While we have some understanding of why random matrix theory leads to the correct answer in function fields, for number fields it is just an observed result that these predictions are useful in guessing the correct behavior.

We cannot stress enough how important it is to have a conjectured answer when studying a difficult problem. Random matrix theory has been a powerful tool in providing conjectures to guide researchers; however, it does have some drawbacks. One of the most severe problems is that random matrix theory fails to incorporate the arithmetic of the problem, which has to be incorporated somehow in order to obtain a correct, complete prediction. This omission was keenly felt in Keating and Snaith's \cite{KeSn1,KeSn2} investigations of the moments of $L$-functions, where the main terms of number theory and random matrix theory differ by arithmetical factors which must be incorporated in a somewhat ad hoc manner into the random matrix predictions.

One approach to such difficulties is the hybrid model of Gonek, Hughes and Keating \cite{GHK}, which is a useful tool for thinking about the interaction between random matrix theory and number theory. They replace an $L$-function with a product of two terms, the first being a truncated Euler product over primes (which has the arithmetic) and the second being a truncated Hadamard product over zeros of the $L$-function (which is modeled by random matrix theory). This model has enjoyed some success; in some cases its predictions can be proved correct, and in other cases its predictions agree with standard conjectures (see \cite{GHK,FGH}).

In this paper we explore another method, the $L$-functions Ratios Conjecture of Conrey, Farmer and Zirnbauer \cite{CFZ1,CFZ2}, which is an extension of the algorithm of Conrey, Farmer, Keating, Rubinstein and Snaith \cite{CFKRS} for computing explicit expressions for the main and lower terms for the moments of $L$-functions.

Frequently a problem in number theory can be reduced to a problem about a family of $L$-functions. The first such instance is Dirichlet's theorem for primes in arithmetic progression, where to count $\pi_{q,a}(x)$ (the number of primes at most $x$ congruent to $a$ modulo $q$) we must understand the properties of $L(s,\chi)$ for all characters $\chi$ modulo $q$. They develop a recipe for conjecturing the value of the quotient of products of $L$-functions averaged over a family, such as
\be \sum_{f \in \mathcal{F}} \frac{L(s+\alpha_1,f) \cdots L(s+\alpha_K,f) L(s+\beta_{1},\overline{f}) \cdots L(s+\beta_{L},\overline{f})}{L(s+\gamma_1,f) \cdots L(s+\gamma_Q,f) L(s+\delta_{1},\overline{f}) \cdots L(s+\delta_{R},\overline{f})} \ee  (we describe their recipe in detail in \S\ref{sec:ratiosrecipe}). Numerous quantities in number theory can be deduced from good estimates of sums of this form; examples include spacings between zeros, $n$-level correlations and densities, and moments of $L$-functions to name just a few. The Ratios Conjecture's answer is expected to be accurate to an error of the order of the square-root of the family's cardinality. This is an incredibly detailed and specific conjecture; to appreciate the power of its predictions, it is worth noting that the standard random matrix theory models cannot predict lower order terms of size $1/\log |\mathcal{F}|$, while the Ratios Conjecture is predicting all the terms down to $O(|\mathcal{F}|^{-1/2+\gep})$.

In this paper we test the predictions of the Ratios Conjecture for the 1-level density of the family of Dirichlet characters of prime conductor $q\to\infty$. The $1$-level density for a family $\mathcal{F}$ of $L$-functions is
\begin{eqnarray}
D_{1,\mathcal{F}}(\phi)\ :=\ \frac{1}{|\mathcal{F}|} \sum_{f\in
\mathcal{F}} \sum_{\ell} \phi\left(\gamma_{f,\ell}\frac{\log
Q_f}{2\pi}\right),
\end{eqnarray} where $\phi$ is an even Schwartz test function
whose Fourier transform has compact support, $\foh +
i\gamma_{f,\ell}$ runs through the non-trivial zeros of $L(s,f)$ (if GRH holds, then each $\gamma_{f,\ell} \in \R$), and $Q_f$ is the analytic conductor of $f$; we see in \S\ref{sec:diffandcontourint} that the $1$-level density equals a contour integral of the derivative of a sum over our family of ratios of $L$-functions. As $\phi$ is an even Schwartz function, most of the contribution to $D_{1,\mathcal{F}}(\phi)$ arises from the zeros  near the central point; thus this statistic is well-suited to investigating the
low-lying zeros.

The 1-level density has enjoyed much popularity recently. The reason is twofold. First, of course, there are many problems where the behavior near the central point is of great interest (such as the Birch and Swinnerton-Dyer Conjecture), and thus we want a statistic relevant for such investigations. The second is that for any automorphic cuspidal $L$-function, the $n$-level correlation of the zeros high up on the critical line (and thus the spacing between adjacent normalized zeros) is conjectured to agree with the Gaussian unitary ensemble from random matrix theory (see \cite{Hej,Mon2,RS} for results for suitably restricted test functions), as well as the classical compact groups \cite{KaSa1,KaSa2}. This leads to the question of what is the correct random matrix model for the zeros of an $L$-function, as different ensembles give the same answer. This universality of behavior is broken if instead of studying zeros high up on the critical line for a given $L$-function we instead study zeros near the central point. Averaging over a family of $L$-functions (whose behaviors are expected to be similar near the central point), the universality is broken, and Katz and Sarnak conjecture that families of $L$-functions correspond to classical compact groups (and the classical compact groups (unitary, symplectic and orthogonal) have different behavior). Specifically, for an infinite family of $L$-functions let $\mathcal{F}_N$ be the subset whose conductors equal $N$ (or are at most $N$). They conjecture that \be \lim_{N\to\infty} D_{1,\mathcal{F}_N}(\phi) \to \int \phi(x) W_{G(\mathcal{F})}(x)dx,\ee where $G(\mathcal{F})$ indicates unitary, symplectic or orthogonal (${\rm SO(even)}$ or ${\rm SO(odd)}$) symmetry. We record the different densities for each family. As \be\int f(x) W_{G(\mathcal{F})}(x)dx\ =\ \int \widehat{f}(u) \widehat{W_{G(\mathcal{F})}}(u)du,\ee it suffices to state the Fourier Transforms. Letting  $\eta(u)$ be $1$ ($1/2$ and $0$) for $|u|$ less than $1$ (equal to $1$ and greater than $1$), and $\delta_0$ the standard Dirac Delta functional, the following are the Fourier transforms of the densities of the scaling limits of the various classical compact groups: \bi \item ${\rm SO(even)}$: $\delta_0(u) + \foh \eta(u)$; \item orthogonal:
$\delta_0(u) + \foh$; \item ${\rm SO(odd)}$: $\delta_0(u) - \foh
\eta(u) + 1$; \item symplectic: $\delta_0(u) -
\foh \eta(u)$; \item unitary: $\delta_0(u)$. \ei Note that the first three densities agree for
$|u| < 1$ and split (ie, become distinguishable) for $|u| \geq 1$, and for any support we can distinguish unitary, symplectic and orthogonal symmetry.

There are now many examples where the main term in 1-level density calculations in number theory agrees with the Katz-Sarnak conjectures (at least for suitably restricted test functions), such as all Dirichlet characters, quadratic Dirichlet characters, $L(s,\psi)$ with $\psi$ a character of the ideal class group of the imaginary quadratic field $\mathbb{Q}(\sqrt{-D})$ (as well as other number fields), families of elliptic curves, weight $k$ level $N$ cuspidal newforms, symmetric powers of ${\rm GL}(2)$ $L$-functions, and certain families of ${\rm GL}(4)$ and ${\rm
GL}(6)$ $L$-functions (see \cite{DM1,DM2,FI,Gu,HR,HuMil,ILS,KaSa2,Mil1,MilPe,OS2,RR,Ro,Rub1,Yo2}).

Now that the main terms have been shown to agree, it is natural to look at the lower order terms (see \cite{FI,HKS,Mil2,Mil4,MilPe,Yo1} for some examples). We give two applications of these terms. Initially the zeros of $L$-functions high on the critical line were modeled by the $N\to\infty$ scaling limits of $N\times N$ complex Hermitian matrices. Keating and Snaith \cite{KeSn1,KeSn2} showed that a better model for zeros at height $T$ is given by $N\times N$ matrices with $N \sim \log(T/2\pi)$ (this choice makes the mean spacing between zeros and eigenvalues equal), and further improvement occurs when we incorporate the lower order terms. Our second example involves bounds on the average rank of one-parameter families of elliptic curves; incorporating the lower order terms (which can have a significant contribution when the conductors are small) leads to slightly higher bounds on the average rank, which are more in line with the observed excess rank (see \cite{Mil2}).

%Even better agreement (see \cite{BBLM}) has been found by replacing $N$ with $N_{\rm effective}$, where the first order correction terms are used to slightly adjust the size of the matrix (as $N\to\infty$, $N_{\rm effective}/N \to 1$).

While the main terms in the $1$-level densities studied to date are independent of the arithmetic of the family, this is not the case for the lower order terms. For example, in \cite{Mil4} differences are seen depending on whether or not the family of elliptic curves has complex multiplication, or what its torsion group is, and so on. Additionally, while the main term of  the one-level density of $L$-functions attached to number fields is independent of properties of the number field, these features can surface in the first lower order term \cite{FI,MilPe}. While random matrix theory is unable to make any predictions about these lower order terms, the Ratios Conjecture gives very detailed statements. These have been verified as accurate (up to square-root agreement as predicted!) for suitably restricted test functions for orthogonal families of cusp forms \cite{Mil5,MilMo} and the symplectic families of Dirichlet characters \cite{Mil3,St}. Further, these lower order terms in families of quadratic twists of a fixed elliptic curve \cite{HKS,HuyMil} have been used as inputs in other problems, such as modeling the first zero above the central point for certain families of elliptic curves \cite{DHKMS}.

The purpose of this paper is to test these predictions for the unitary family of Dirichlet characters. We review some needed properties of these $L$-functions in \S\ref{sec:reviewpropLfnsDChars} and then state our results in \S\ref{sec:statementresults}.

\subsection{Review of Dirichlet $L$-functions}\label{sec:reviewpropLfnsDChars}

We quickly review some needed facts about Dirichlet characters and $L$-functions; see
\cite{Da,IK} for details. Let $\chi$ be a non-principal
Dirichlet character of prime modulus $q$. Let $\tau(\chi)$ be the Gauss
sum \be \tau(\chi) \ := \ \sum_{k=1}^{q-1} \chi(k) e(k/q),
\ee which is of modulus $\sqrt{q}$; as always, throughout the paper we use \be e(z)\ =\ e^{2\pi i z}.\ee Let \be L(s,\chi) \ := \
\prod_p \left(1 -\chi(p)p^{-s}\right)^{-1} \ee be the $L$-function
attached to $\chi$; the completed $L$-function is \be
\Lambda(s,\chi) \ := \ \left(\frac{\pi}{q}\right)^{-(s+a(\chi))/2} \gam{s+a(\chi)}{2}
L(s,\chi) \ = \  \frac{\tau(\chi)}{i^{a(\chi)/2}\sqrt{q}}
\Lambda(1-s,\overline{\chi}), \ee where \be\label{eq:defnachi} \twocase{a(\chi)
 \ := \ }{0}{if $\chi(-1) = 1$}{1}{if $\chi(-1) =
-1$.}\ee We write the non-trivial zeros of $\Lambda(s,\chi)$ as $\foh+i\gamma$; if we
assume GRH then $\gamma\in\R$. We have \bea\label{eq:logderivDir}
\frac{\Lambda'(s,\chi)}{\Lambda(s,\chi)}  &=& \frac{\log\frac{q}{\pi}}2 +
\foh \frac{\Gamma'}{\Gamma}\left(\frac{s+a(\chi)}{2}\right)  + \frac{L'(s,\chi)}{L(s,\chi)} \ = \ -\frac{\Lambda'(1-s,\chi)}{\Lambda(1-s,\chi)}, \eea which implies \bea\label{eq:negLprimeoverL} -\frac{L'(1-s,\chi)}{L(1-s,\chi)}= \frac{L'(s,\chi)}{L(s,\chi)} + \log \frac{q}{\pi} + \frac12 \gam{1-s+a(\chi)}{2} + \frac12\gam{s+a(\chi)}{2}.\ \ \ \eea

We study $\mathcal{F}(q)$, the family of non-principal characters modulo a prime $q$ (which will tend to infinity). For each $q$, $|\mathcal{F}(q)| = q-2$. The following lemma is the starting point for the analysis of the sums in the $1$-level density.

\begin{lem}\label{lem:sumsdirichletcharinFq} For $q$ a prime, \be \twocase{\sum_{\chi \in \mathcal{F}(q)} \chi(r) \ = \ -1\ + \ }{q-1}{if $r \equiv 1 \bmod q$}{0}{otherwise.} \ee
\end{lem}

\begin{proof} This follows immediately from the orthogonality relations of Dirichlet characters. If $\chi_0$ denotes the principal character, $\chi_0(r) = 0$ if $r \equiv 0 \bmod q$ and $1$ otherwise; the lemma now follows from the well-know relation \be \twocase{\sum_{\chi \bmod q} \chi(r) \ = \ }{q-1}{if $r \equiv 1 \bmod q$}{0}{otherwise.} \ee
\end{proof}

\subsection{Results}\label{sec:statementresults}

Our first result (Theorem \ref{thm:expansionRDalphagamma}) is the Ratios Conjecture's prediction for the sum over $\mathcal{F}(q)$ of the quotient of $L$-functions. The 1-level density can be recovered by a contour integral of its derivative, which we do in Theorem \ref{thm:mainratiospred}. We then compare this prediction to what can be proved in number theory (Theorem \ref{thm:mainNT}). We end the introduction by discussing how standard number theory conjectures lead to extending the support in Theorem \ref{thm:mainNT}, and the extensions agree with the Ratios Conjecture prediction.

\begin{thm}\label{thm:expansionRDalphagamma} Let $q$ be a prime number, and set
\bea R_{\mathcal{F}(q)}(\alpha,\gamma) & \ :=\ &  \sum_{\chi \in \mathcal{F}(q)} \frac{L(1/2 + \alpha, \chi)}{L(1/2 + \gamma, \chi)} \nonumber\\ G_\pm(\alpha) & \ :=\ &  \frac{\Gamma \left(\frac{3}{4}-\frac{\alpha}{2}\right)}{i\Gamma \left(\frac{3}{4}+\frac{\alpha}{2}\right)} \pm \frac{\Gamma \left(\frac{1}{4}-\frac{\alpha}{2}\right)}{\Gamma \left(\frac{1}{4}+\frac{\alpha}{2}\right)}. \eea The Ratios Conjecture prediction for $R_{\mathcal{F}(q)}(\alpha,\gamma)$ is
\bea\label{eq:ratiospredRU} R_{\mathcal{F}(q)}(\alpha,\gamma) & \ = \ & (q-1) \sum_{h\equiv 1 \bmod q} \frac{\mu(h)}{h^{1/2+\gamma}}+(q-1)\left[\frac{G_+(\alpha) e(-1/q)}{2q^{1/2+\alpha}}  + \frac{G_-(\alpha) e(1/q)}{2q^{1/2+\alpha}} \right] \nonumber\\ & & \ \ - \ \frac{\zeta\left(\frac12+\alpha\right)}{\zeta\left(\frac12+\gamma\right)} - \frac{G_+(\alpha)}{2q^{1/2+\alpha}} \frac{\zeta\left(\foh-\alpha\right)}{\zeta\left(\foh+\gamma\right)}- \frac{G_-(\alpha)}{2q^{1/2+\alpha}}\frac{\zeta\left(\foh - \alpha\right)}{\zeta\left(\foh+\gamma\right)} \nonumber\\ & & \ \ + \ O\left(q^{1/2+\gep}\right). \eea The bracketed term is not present if we use the standard Ratios Conjecture, but is present if instead we study a weaker variant  (in the weaker version, we do not drop all terms multiplied by the sign of the functional equation when the signs average to zero, but instead analyze these terms). This distinction is immaterial for our purposes, as the 1-level density involves the derivative, and in both cases these contribute $O(q^{-1/2+\epsilon})$ after we divide by the cardinality of the family.
\end{thm}

\begin{thm}[Ratios Conjecture Prediction]\label{thm:mainratiospred} Denote the $1$-level density for $\mathcal{F}(q)$ (the family of non-principal Dirichlet characters modulo a prime $q$) by \be
D_{1,\mathcal{F}(q)}(\phi) \ := \ \frac1{q-2}\sum_{\chi\in \mathcal{F}(q)}
\sum_{\gamma_\chi \atop L(1/2+i\gamma_\chi,\chi) = 0} \phi\left(\gamma_\chi
\frac{\log \frac{q}{\pi}}{2\pi}\right), \ee with $\phi$ an even Schwartz function whose Fourier transform has compact support.  The Ratios Conjecture's prediction for the $1$-level density of the family of non-principal Dirichlet characters modulo $q$ is \bea D_{1,\mathcal{F}(q)}(\phi) & \ = \ &  \hphi(0) + \frac1{(q-2)\log \frac{q}{\pi}} \sum_{\chi\in\mathcal{F}(q)} \int_{-\infty}^\infty \phi(\tau) \left[ \frac{\Gamma'}{\Gamma}\left(\frac14 + \frac{a(\chi)}{2}+\frac{\pi i \tau}{\log \frac{q}{\pi}}\right) \right]d\tau \nonumber\\ & & \ \ \ \ + \ O\left(q^{-1/2+\gep}\right),\eea with $a(\chi)$ as in \eqref{eq:defnachi}.
\end{thm}

\begin{thm}[Theoretical Results]\label{thm:mainNT} Notation as in Theorem \ref{thm:mainratiospred}, let $\mathcal{F}(q)$ denote the family of non-principal characters to a prime modulus $q$ ($|\mathcal{F}(q)| = q-2$) and $\phi$ an even Schwartz functions such that $\supp(\hphi) \subset (-\sigma, \sigma)$ for any $\sigma < 2$. Then \bea  D_{1,\mathcal{F}(q)}(\phi)& \ = \ &  \widehat{\phi}(0) + \frac{1}{(q-2)\log \frac{q}{\pi}} \int_{-\infty}^\infty \phi(\tau)\sum_{\chi \in \mathcal{F}(q)} \left[\frac{\Gamma'}{\Gamma}\left(\frac{1}{4}+\frac{a(\chi)}{2}+\frac{\pi i \tau}{\log \frac{q}{\pi}}\right)\right]\,d\tau \nonumber \\
     & & \ \ \ \ \ \ \ \ \ + \ O\left(q^{\frac{\sigma}2-1+\gep}\right). \eea
\end{thm}

We note that we have agreement up to square-root cancelation in the family's cardinality in Theorems \ref{thm:mainratiospred} and \ref{thm:mainNT}, provided that $\supp(\hphi) \subset (-1,1)$; if instead the support is contained in $(-2,2)$ then we have agreement up to a power savings. Unlike previous tests of the Ratios Conjecture, in this case the Ratios Conjecture prediction does not have a lower order term given by an Euler product. This is not surprising, as we expect the $1$-level density to essentially be just $\hphi(0)$ (the integral term that we find arises in a natural way from the Gamma factors in the functional equation; if we were to slightly modify our normalization of the zeros then we could remove this term). One of the most important consequences of the Ratios Conjecture is that it predicts this should be the answer for arbitrary support, though we can only prove it (up to larger error terms) for support in $(-2,2)$. In \cite{Mil6} it is shown how to extend the support in the theoretical results up to $(-4,4)$ or even any arbitrarily large support \emph{if} we assume standard conjectures about how the error terms of primes in arithmetic progression depend on the modulus. Thus the Ratios Conjecture's prediction becomes another way to test the reasonableness of some standard number theory conjectures.

For example, consider the error term in Dirichlet's theorem for primes in arithmetic progression. Let $\pi_{q,a}(x)$ denote the number of primes at most $x$ that are congruent to $a$ modulo $q$. Dirichlet's theorem says that, to first order, $\pi_{q,a}(x) \sim \pi(x)/\phi(q)$. We set $E(x;q,a)$ equal to the difference between the observed and predicted number of primes: \be E(x;q,a) \ := \ \left|\pi_{q,a}(x) - \frac{\pi(x)}{\varphi(q)}\right|. \ee We have \be E(x;q,a) \ = \ O(x^{1/2} (qx)^\gep) \ee under GRH. We expect the error term to have some $q$-dependence; the philosophy of square-root cancelation suggests $(x/q)^{1/2} (qx)^\gep$. Montgomery \cite{Mon1} conjectured bounds of this nature. Explicitly, assume \begin{conj}\label{conj:weakmont} There is a
$\theta \in [0,\foh)$ such that for $q$ prime \be E(x;q,a) \ \ll \
q^\theta \cdot \sqrt{\frac{x}{\varphi(q)}} \cdot (xq)^\gep. \ee \end{conj} Combining the number theory calculations of Miller in \cite{Mil6} (which assume Conjecture \ref{conj:weakmont}  in the special case of $a=1$) with the Ratios Conjecture calculations in this paper, we find number theory and the Ratios Conjecture prediction agree for arbitrary finite support with a power savings.

Alternatively, consider the following conjecture:

\begin{conj}\label{conj:miller}  There exists
an $\eta \in [0,1)$ such that for prime $q$, \be\label{eq:conjmill2b} E(x;q,1)^2 \ \ll \ q^\eta \cdot \frac1{q}
\sum_{a = 1 \atop (a,q) = 1}^m E(x;q,a)^2. \ee \end{conj}

Again combining the results of Miller \cite{Mil6} with this work, we find agreement between number theory and the Ratios Conjecture prediction, though this time only for test functions with $\supp(\hphi) \subset (-4+2\eta, 4-2\eta)$. Thus the Ratios Conjecture may be interpreted as providing additional evidence for these conjectures.

\begin{rek} These two conjectures are quite reasonable. The bound in the first is true when $\theta = 1/2$ by GRH. The bound in the second is trivially true when $\eta = 1$, as the error on the left side is then contained in the sum on the right. We expect the first to hold for $\theta=0$ and the second for $\eta = \epsilon$. \end{rek}

\begin{rek} Similar to \cite{HR}, we assume $q$ is prime (see \cite{Mil6} for details on how to remove this assumption in the theoretical results). Our purpose in this paper is to describe the Ratios Conjecture's recipe and show agreement between its prediction and number theory, highlighting the new features that arise in this test of the Ratios Conjecture which have not surfaced in other investigations. We therefore assume $q$ is prime for ease of exposition, as it simplifies some of the arguments. \end{rek}

The paper is organized as follows. We describe the Ratios Conjecture's recipe in \S\ref{sec:ratiosconj} and prove Theorems \ref{thm:expansionRDalphagamma} and \ref{thm:mainratiospred}. We then prove Theorem \ref{thm:mainNT} in the following section. There are obviously similarities between the computations in this paper and those in \cite{Mil3}, where the family of quadratic Dirichlet characters was studied. The computations there, at times, were deliberately done in greater generality than needed, and thus we refer the reader to \cite{Mil3} for details at times (such as the proof of the explicit formula).

%%%%%%%%%%%%%%%%%%%%%%%%%%%%%%%%%%%%%%%%%%%%%%%%%%%%%%%%%%%%%%%%%%%%%%%%%%%%%%%%%%%%%%%%%%%%%%%%
%%%%%%%%%%%%%%%%%%%%%%%%%%%%%%%%%%%%%%%%%%%%%%%%%%%%%%%%%%%%%%%%%%%%%%%%%%%%%%%%%%%%%%%%%%%%%%%%
%%%%%%%%%%%%%%%%%%%%%%%%%%%%%%%%%%%%%%%%%%%%%%%%%%%%%%%%%%%%%%%%%%%%%%%%%%%%%%%%%%%%%%%%%%%%%%%%
\section{Ratios Conjecture}\label{sec:ratiosconj}

\subsection{Recipe}\label{sec:ratiosrecipe}

We follow the recipe of the Ratios Conjecture and state its prediction for the 1-level density of the family of non-principal, primitive Dirichlet characters of prime modulus $q \to \infty$. We denote this family by $\mathcal{F}(q)$, and note $|\mathcal{F}(q)| = q-2$.

The Ratios Conjecture concerns estimates for \be R_{\mathcal{F}(q)}(\alpha,\gamma)\ :=\ \sum_{\chi \in \mathcal{F}(q)} \frac{L(1/2 + \alpha, \chi)}{L(1/2 + \gamma, \chi)}; \ee \emph{the convention is \textbf{not} to divide by the family's cardinality}. The conjectured formulas are believed to hold up to errors of size $O(|\mathcal{F}(q)|^{1/2+\gep})$. We briefly summarize how to use the Ratios conjecture to predict answers; for more details see \cite{CFZ1}, as well as \cite{CS,Mil3}.

\ben

\item Use the approximate functional equation to expand the numerator into two sums plus a remainder. The first sum is over $m$ up to $x$ and the second over $n$ up to $y$, where $xy$ is of the same size as the analytic conductor (typically one takes $x \sim y \sim \sqrt{q}$). We ignore the remainder term.

\item Expand the denominator by using the generalized Mobius function.

\item Execute the sum over $\mathcal{F}(q)$, replacing each summand by the diagonal term in its expected value when averaged over the family; however, before executing these sums replace any product over epsilon factors (arising from the signs of the functional equations) with the average value of the sign of the functional equation in the family. One may weaken the Ratios Conjecture by not discarding these terms; this is done in \cite{Mil5,MilMo}, where as predicted it is found that these terms do not contribute. To provide a better test, we also do not drop these terms (see Remark \ref{rek:weakerratios} for a discussion of which terms, for this family, may be ignored).

\item Extend the $m$ and $n$ sums to infinity (i.e., complete the products).

\item Differentiate with respect to the parameter $\alpha$, and note that the size of the error term does not significantly change upon differentiating. There is no error in this step, which can be justified by elementary complex analysis because all terms under consideration are analytic. See Remark 2.2 of \cite{Mil5} for details.

\item A contour integral involving $\frac{\partial}{\partial \alpha}R_{\mathcal{F}(q)}(\alpha,\gamma)\Big|_{\alpha=\gamma=s}$ yields the 1-level density, where the integration is on the line $\Re(s) = c > 1/2$.

\een

\subsection{Approximate Functional Equation and Mobius Inversion}

We now describe the steps in greater detail. The approximate functional equation (see for example \cite{IK})  states
\bea L\left(\frac{1}{2} + \alpha, \chi\right) & \ =\ & \sum_{n \leq x} \frac{\chi(n)}{n^{1/2+\alpha}} + \frac{\tau(\chi)}{i^{a(\chi)}q^{\foh}}\ q^{\foh-s} \frac{\Gamma \left(\frac{1}{4}-\frac{\alpha}{2}+\frac{a(\chi)}{2}\right)}{\Gamma \left(\frac{1}{4}+\frac{\alpha}{2}+\frac{a(\chi)}{2}\right)} \sum_{m \leq y} \frac{\overline{\chi}(m)}{m^{1-s}} \nonumber\\ & & \ \ \ + \ {\rm Error}, \eea
where \be \twocase{a(\chi) \ = \ }{0}{if $\chi(-1) = 1$}{1}{if $\chi(-1) =
-1$}\ee and \be \tau(\chi) \ = \ \sum_{x \bmod m} \chi(x) e(x/m) \ee is the Gauss sum (which is of modulus $\sqrt{m}$ for $\chi$ non-principal). We ignore the error term in the approximate functional equation when we expand $L(1/2+\alpha,\chi)$ in our analysis of $R_{\mathcal{F}(q)}(\alpha,\gamma)$.

By Mobius Inversion we have
\be \frac{1}{L(\frac{1}{2} + \gamma, \chi)}\ =\ \sum_{h=1}^{\infty} \frac{\mu(h) \chi (h)}{h^{1/2 + \gamma}}\ee where \be \threecase{\mu(h) \ = \ }{1}{if $h=1$}{(-1)^r}{if $h = p_1\cdots p_r$ is the product of $r$ distinct primes}{0}{otherwise.} \ee  

We combine the above to obtain an expansion for $R_{\mathcal{F}(q)}(\alpha,\gamma)$. Note that $R_{\mathcal{F}(q)}(\alpha,\gamma)$ involves evaluating the $L$-functions at $1/2+\alpha$ and $1/2+\gamma$; thus the $q^{1/2-s}/q^{1/2}$ term is just $q^{-1/2-\alpha}$ when $s=1/2+\alpha$.
\bea\label{eq:firstexpansionRDalphagamma}
& & R_{\mathcal{F}(q)} (\alpha,\gamma)\nonumber\\ & = & \ \sum_{\chi} \sum_{h=1}^{\infty} \frac{\mu(h) \chi (h)}{h^{1/2 + \gamma}} \left(\sum_{n \leq x} \frac{\chi(n)}{n^{1/2+\alpha}} + \frac{\tau(\chi)}{i^{a(\chi)}q^{\foh+\alpha}} \frac{\Gamma \left(\frac{1}{4}-\frac{\alpha}{2}+\frac{a(\chi)}{2}\right)}{\Gamma \left(\frac{1}{4}+\frac{\alpha}{2}+\frac{a(\chi)}{2}\right)} \sum_{m \leq y} \frac{\overline{\chi}(m)}{m^{1-s}}\right) \nonumber\\  & =&  \ \sum_{n \leq x}  \sum^{\infty}_{h=1}  \sum_{\chi \in \mathcal{F}(q)} \frac{\mu(h) \chi(nh)}{n^{\foh+\alpha} h^{\foh + \gamma}} +  \sum_{\chi} \frac{\tau(\chi)}{i^{a(\chi)}q^{\frac{1}{2} +\alpha}}\frac{\Gamma \left(\frac{1}{4}-\frac{\alpha}{2}+\frac{a(\chi)}{2}\right)}{\Gamma \left(\frac{1}{4}+\frac{\alpha}{2}+\frac{a(\chi)}{2}\right)}  \sum^{\infty}_{h=1}  \sum_{m \leq y}  \frac{\mu(h)\chi(h) \overline{\chi}(m)}{m^{\frac{1}{2} - \alpha}h^{\frac{1}{2} + \gamma}} \nonumber\\ \eea

\begin{rek}\label{rek:weakerratios} If we assume the standard form of the Ratios Conjecture, we may ignore the contribution from the second piece above. This is because the signs of the functional equations essentially average to zero, and thus according to the recipe there is no contribution from these terms. To see this, note the sign of the functional equation is $\tau(\chi) / i^{a(\chi)} q^{1/2}$. We have \be i^{-a(\chi)}\ =\ \frac{\chi(-1)+1}2 + \frac{\chi(-1)-1}{2}\ i. \ee Thus, expanding the Gauss sum, we see it suffices to show sums such as \be \mathcal{C} \ = \ \frac{1}{q-2}\sum_{\chi \in \mathcal{F}(q)} \sum_{x \bmod m} \frac{\chi(\pm x) \exp(2\pi i x / q)}{q^{1/2}} \cdot \frac{i^{1\pm 1}}{2} \ee are small. We may extend the summation to include the principal character at a cost of $O(q^{-3/2})$ (as the sum over $x$ is $-1$ for the principal character). We now have a sum over all characters, with $\sum_{\chi \bmod q} \chi(\pm x) = q-1$ if $\pm x \equiv 1 \bmod q$ and 0 otherwise. Thus we find \be \mathcal{C} \ = \ \frac{i^{1\pm 1} \exp(\pm 2\pi i /q)}{2q^{1/2}} \frac{q-1}{q-2}  + O(q^{-3/2}); \ee as this is of size $q^{-1/2}$, it is essentially zero and thus, according to the Ratios recipe, it should be ignored. We choose not to ignore these terms to provide a stronger test of the Ratios Conjecture.
\end{rek}

\subsection{Executing the sum over $\mathcal{F}(q)$ and completing the sums}

Returning to \eqref{eq:firstexpansionRDalphagamma}, we want to pass the summation over $\chi$ through everything to the product of the expansion of $\tau(\chi)$ as a character sum and the $\chi(nh)$ and $\chi(h)\overline{\chi}(m)$ terms below (note, as explained in Remark \ref{rek:weakerratios}, we may drop these terms if we assume the standard Ratios Conjecture; we desire a stronger test and thus we will partially analyze these terms). Unfortunately the Gamma factors and the $i^{-a(\chi)}$ factor in the sign of the functional equation depend on $\chi$. Fortunately this dependence is weak, as $a(\chi) = 0$ if $\chi(-1) = 1$ and $-1$ otherwise. To facilitate summing over the characters we introduce factors $\frac{\chi(-1)+1}2$ and $\frac{\chi(-1)-1}2$ below, giving
\bea 	& & R_{\mathcal{F}(q)} (\alpha,\gamma) \ = \ \sum_{n \leq x}  \sum^{\infty}_{h=1}  \sum_{\chi \in \mathcal{F}(q)} \frac{\mu(h) \chi(nh)}{n^{\foh+\alpha} h^{\foh + \gamma}} \nonumber\\
										&	& \ \ + \  \sum_{\chi}\Bigg[\frac{\chi(-1)+1}{2} \frac{\tau(\chi)}{iq^{\frac{1}{2} + \alpha}}\frac{\Gamma \left(\frac{3}{4}-\frac{\alpha}{2}\right)}{\Gamma \left(\frac{3}{4}+\frac{\alpha}{2}\right)}  \sum^{\infty}_{h=1}  \sum_{m \leq y}  \frac{\mu(h)\chi(h) \overline{\chi}(m)}{m^{\frac{1}{2} - \alpha}h^{\frac{1}{2} + \gamma}}  \nonumber \\
											& & \ \ +\  \frac{\chi(-1)-1}{2} \frac{\tau(\chi)}{q^{\frac{1}{2} + \alpha}}\frac{\Gamma \left(\frac{1}{4}-\frac{\alpha}{2}\right)}{\Gamma \left(\frac{1}{4}+\frac{\alpha}{2}\right)}  \sum^{\infty}_{h=1}  \sum_{m \leq y}  \frac{\mu(h)\chi(h) \overline{\chi}(m)}{m^{\frac{1}{2} - \alpha}h^{\frac{1}{2} + \gamma}} \Bigg].
\eea Distributing and regrouping yields
\begin{align}\label{eq:rewrittenRDalphagamma} &R_{\mathcal{F}(q)} (\alpha,\gamma) \ = \ \sum_{n \leq x}  \sum^{\infty}_{h=1}  \sum_{\chi \in \mathcal{F}(q)} \frac{\mu(h) \chi(nh)}{n^{\foh+\alpha} h^{\foh + \gamma}}\nonumber\\
&\hspace{3 mm}\ +\left(\frac{\Gamma \left(\frac{3}{4}-\frac{\alpha}{2}\right)}{i\Gamma \left(\frac{3}{4}+\frac{\alpha}{2}\right)} + \frac{\Gamma \left(\frac{1}{4}-\frac{\alpha}{2}\right)}{\Gamma \left(\frac{1}{4}+\frac{\alpha}{2}\right)}\right)\sum_{\chi \in \mathcal{F}(q)} \frac{\chi(-1)\tau(\chi)}{2q^{\frac{1}{2}+\alpha}}\sum_{h=1}^\infty\sum_{m\leq y} \frac{\mu(h)\chi(h)\bar{\chi}(m)}{m^{\frac{1}{2}-\alpha}h^{\frac{1}{2}+\gamma}}\nonumber\\
&\hspace{3 mm}\ +\left(\frac{\Gamma \left(\frac{3}{4}-\frac{\alpha}{2}\right)}{i\Gamma \left(\frac{3}{4}+\frac{\alpha}{2}\right)} - \frac{\Gamma \left(\frac{1}{4}-\frac{\alpha}{2}\right)}{\Gamma \left(\frac{1}{4}-\frac{\alpha}{2}\right)}\right)\sum_{\chi \in \mathcal{F}(q)} \frac{\tau(\chi)}{2q^{\frac{1}{2}+\alpha}}\sum_{h=1}^\infty\sum_{m\leq y}\frac{\mu(h)\chi(h)\bar{\chi}(m)}{m^{\frac{1}{2}-\alpha}h^{\frac{1}{2}+\gamma}} \nonumber\\ & = \ \ \mathcal{S}_1 + \mathcal{S}_2 + \mathcal{S}_3, \end{align} \emph{where again \emph{only} the first term $\mathcal{S}_1$ is present if we assume the strong form of the Ratios Conjecture.}

The proof of Theorem \ref{thm:expansionRDalphagamma} follows immediately from the above expansion and Lemmas \ref{lem:ratiosconjmathcals1} and \ref{lem:S2S3computationlemma}.

%\sum_{h \equiv 1 \bmod q} \frac{\mu(h)}{h^{\frac12+\gamma}}

\begin{lem}\label{lem:ratiosconjmathcals1} The Ratios Conjecture's recipe predicts \be \mathcal{S}_1\ = \ (q-1) \sum_{h \equiv 1 \bmod q} \frac{\mu(h)}{h^{\frac12+\gamma}} - \frac{\zeta\left(\frac12+\alpha\right)}{\zeta\left(\frac12+\gamma\right)} + {\rm small}. \ee  \end{lem}

\begin{proof} By Lemma \ref{lem:sumsdirichletcharinFq}, we have \be \twocase{\sum_{\chi \in \mathcal{F}(q)} \chi(r) \ = \ -1\ + \ }{q-1}{if $r \equiv 1 \bmod q$}{0}{otherwise.} \ee According to Step 3 of the Ratios Conjecture's recipe, we now replace the sum over the family with the diagonal term in its expected value. We are assisted in this by Lemma \ref{lem:sumsdirichletcharinFq}, which gives us explicit formulas for these sums.

%According to the Ratios Conjecture, we should only keep the `diagonal' (i.e., the `main') term in the family sum.
%Unlike the other families investigated (the symplectic family of quadratic characters in \cite{Mil3} or the orthogonal
%families of cuspidal newforms in \cite{Mil5,MilMo}), it is not immediately clear what the Ratios Conjecture means by
%`diagonal'. Clearly we always have a contribution of -1 in summing over the family; however, what do we do about the
%factor of $q-1$? This is a large factor, but it occurs rarely, specifically only when $r \equiv 1 \bmod q$? For now we
%keep this term and analyze the consequences of keeping it below.

We thus find that \bea\mathcal{S}_1 &\ = \ & \sum_{n \leq x}  \sum^{\infty}_{h=1}  \sum_{\chi \in \mathcal{F}(q)} \frac{\mu(h) \chi(nh)}{n^{\foh+\alpha} h^{\foh + \gamma}} \nonumber\\
&=& (q-1)\sum_{nh \equiv 1 (q) \atop n \le x}\frac{\mu(h)}{n^{\foh+\alpha} h^{\foh + \gamma}} - \sum_{n \leq x} \sum_{h=1}^{\infty}\frac{\mu(h)}{n^{\foh+\alpha} h^{\foh + \gamma}}. \eea

The second sum above is readily evaluated after we complete it by sending $x\to\infty$; it is just $\zeta(1/2+\alpha)/\zeta(1/2+\gamma)$. We are of course completely ignoring convergence issues; however, under the Riemann Hypothesis the $h$-sum converges for $\Re(\gamma) > 0$. The $n$-sum is initially finite, and should be replaced with a finite Euler product approximation to the Riemann zeta function; letting $x\to\infty$ gives $\zeta(1/2+\alpha)$. In our application below we will divide by the family's cardinality. Thus this piece will contribute $O(1/q)$ to the 1-level density, and yield a negligible term (in fact, a term significantly smaller than the conjectured error of size $O(q^{-1/2+\epsilon})$.

We now analyze the sum with $nh \equiv 1 \bmod q$, first giving a rough estimate and then a more refined one. While it is multiplied by the large factor $q-1$, it also has the condition $nh\equiv 1 \bmod q$. This congruence greatly lessens the contribution as we have $n$'s and $h$'s in arithmetic progression. Further, we haven't divided by the cardinality of the family (which is of size $q$). Finally, we have the Mobius factor $\mu(h)$ in the numerator. Thus it is reasonable to expect that the part that depends on $\alpha$ and $\gamma$ will be small; in other words, the sum should be well-approximated by the $n=h=1$ term, which gives $q-1$. While this factor is large (it leads to a term of size 1 when we divide by the cardinality of the family), there is \emph{no} dependence on $\alpha$ or $\gamma$. As it is the derivative of $R_{\mathcal{F}(q)}(\alpha,\gamma)$ that arises in our computation of the 1-level density, this large term is actually harmless.

Arguing more carefully, we note that $nh \equiv 1 \bmod q$ ensures that unless $n=h=1$ the product $nh$ will be at least $q$, and thus such terms should be small. We can readily handle the contribution when $n=1$; this is just
\be (q-1) \sum_{h \equiv 1 \bmod q} \frac{\mu(h)}{h^{\frac12+\gamma}}. \ee This term should be $q-1$ plus small for two reasons: after the $h=1$ term all denominators are one more than a multiple of $q$, and sums of the Mobius function in arithmetic progressions should be small \cite{Wa}). For us, however, what matters is that this sum is independent of $\alpha$, and thus when we differentiate with respect to $\alpha$ in a few steps it will vanish, and therefore not contribute to the 1-level density.

We are thus left with the contribution from $nh \equiv 1 \bmod q$ with $2 \le n \le \sqrt{q}$. Arguing in the same spirit as Conrey, Farmer and Zirnbauer \cite{CFZ1} (Section 5.1) and Conrey and Snaith \cite{CS} (page 597), we absorb the contribution of these terms into the error. The reason is that in the papers referenced above, the authors state then when replacing summands with their expected value, only the diagonal term is kept. Specifically, when we let $\alpha$ and $\gamma$ be complex numbers of the form $c+it$, the oscillatory nature of $t$ in the sum (not to mention the Mobius function and the presence of a factor of $q$ in the denominators) leads these terms to contribute at a lower order. Alternatively, we provide another heuristic for why this term should be negligible. Though it is multiplied by a factor of $q-1$, it has on the order of $1/q$ as many terms as the sum without the $nh\equiv 1 \bmod q$ restriction, and thus should contribute at the same order.
\end{proof}

%\textbf{SADLY ARGUMENT BELOW IS SLIGHTLY OFF. PROBLEM IS THAT THE $h$-SUM IS UNBOUNDED. MAYBE EPSTEIN ZETA FUNCTIONS?}
%More rigorously, note that it is not until Step 4 that we are to send the parameters to infinity. In the approximate functional equation, we take $x$ and $y$ to be of size $\sqrt{q}$. As $n \le x$, this means that the only $n$ term which survives is $n=1$. Thus we are left with a sum over $h$ such that $h \equiv 1 \bmod q$, which allows us to write $h = q\ell + 1$. The resulting sum is $O(q^{1/2-\gamma}/\log^A q)$ for any $A$. This follows from partial summation, using classical bounds (see \cite{Wa}) for the sum of the Mobius function in arithmetic progressions, specifically \be \sum_{h \le u \atop h \equiv a \bmod m} \ \ll\ u \exp\left(-C(\log u)^{2/3} / (\log\log u)^{1/5}\right), \ \ \ m \ \le \ (\log u)^B \ee with $C$ and the implied constant depending on $A$.
%The Ratios Conjecture recipe states that, when executing the summation over the family, only the `diagonal' (i.e., the `main') term should be kept. We can think of two ways to interpret this: (1) the $nh\equiv 1 \bmod q$ is not a `diagonal' term, or (2) the $nh\equiv 1 \bmod q$ terms contribute $(q-1) + {\rm small}$. While these two interpretations yield different values for $R_{\mathcal{F}(q)}(\alpha,\gamma)$, they give the same contribution for the derivative, which is all we care about. See also Remark \ref{rek:droppingtermsinmathcalS1} for more reasons why the first sum may safely be ignored.

We isolate two points from the above arguments that will be of use later.

\begin{rek} The important point to note in evaluating $\mathcal{S}_1$ is that, for the purposes of differentiating, the first term is independent of $\alpha$ and thus does not contribute to the one-level density. For the factor of $\zeta(1/2+\alpha)/\zeta(1/2+\gamma)$, as there is \emph{no} $q$-dependence, upon dividing by the cardinality of the family we find a contribution of size $O(1/q)$ to the 1-level density. Returning to the first term, while the factor $q-1$ (arising from $h=1$) is large even upon division by the family's cardinality, it is independent of $\alpha$ and $\gamma$, and thus does not contribute when we execute Step 5, differentiating with respect to the parameters. Note that a factor of this size \emph{must} be present; to see this, consider the special case $\alpha = \gamma$. There $R_{\mathcal{F}(q)}(\alpha,\gamma) = |\mathcal{F}(q)| = q-2$, which is $q-1 + O(q^{1/2+\epsilon})$. \end{rek}

\begin{rek}\label{rek:droppingtermsinmathcalS1} Returning to the analysis of the first piece of $\mathcal{S}_1$, note that $n \le x \sim \sqrt{q}$ means that in each congruence restriction $nh\equiv 1 \bmod q$, there is at most one $n$ that works. In the special case of $h\equiv 1 \bmod q$, this means $n=1$. If $n=2$ then $h \ge (q+1)/2$, and thus the first term in this $h$-arithmetic progression is large. In particular, as $n \le x \sim \sqrt{q}$ we have $h \ge \sqrt{q}$ for $n \ge 2$. All these arguments strongly imply that this sum should be negligible (except perhaps for the $n=h=1$ term, which is constant). \end{rek}

Before analyzing the remaining pieces of \eqref{eq:rewrittenRDalphagamma} (\emph{which are not present if we assume the strong form of the Ratios Conjecture}), it is convenient to set
\be\label{eq:defnGpm} G_\pm(\alpha)\ =\ \frac{\Gamma \left(\frac{3}{4}-\frac{\alpha}{2}\right)}{i\Gamma \left(\frac{3}{4}+\frac{\alpha}{2}\right)} \pm \frac{\Gamma \left(\frac{1}{4}-\frac{\alpha}{2}\right)}{\Gamma \left(\frac{1}{4}+\frac{\alpha}{2}\right)}.\ee

\begin{lem}\label{lem:S2S3computationlemma} We have \bea \mathcal{S}_2  &\ = \ & \frac{(q-1) G_+(\alpha) e(-1/q)}{2q^{1/2+\alpha}} - \frac{G_+(\alpha)}{2q^{1/2+\alpha}} \frac{\zeta\left(\foh-\alpha\right)}{\zeta\left(\foh+\gamma\right)} + {\rm small} \nonumber\\  \mathcal{S}_3 & \ = \ & \frac{(q-1)G_-(\alpha) e(1/q)}{2q^{1/2+\alpha}}  - \frac{G_-(\alpha)}{2q^{1/2+\alpha}}\frac{\zeta\left(\foh - \alpha\right)}{\zeta\left(\foh+\gamma\right)} + {\rm small}. \eea \end{lem}

%where, similar to Lemma \ref{lem:ratiosconjmathcals1}, depending on how we interpret the Ratios Conjecture's recipe of keeping only the `main' terms the first term in the expansion for $\mathcal{S}_2$ and $\mathcal{S}_3$ above may or may not be present (for our purposes, this won't matter as both are $O(q^{-1/2+\epsilon})$ after differentiation and division by the cardinality of the family).

\begin{proof} Essentially the only difference between the analysis of $\mathcal{S}_2$ and $\mathcal{S}_3$ is that $\mathcal{S}_2$ has (effectively) $\chi(-h)$ instead of $\chi(h)$. We therefore just remark on the minor changes needed to evaluate $\mathcal{S}_2$ after evaluating $\mathcal{S}_3$.

Let $e(z) = \exp(2\pi i z)$. Using the expansion for the Gauss sum $\tau(\chi)$ (when we expand it below we start the sum at $a=1$ and not $a=0$ as $\chi(0) = 0$) we find
\bea \mathcal{S}_3 & \ = \ & \frac{G_-(\alpha)}{2} \sum_{m \leq y} \sum_{h=1}^{\infty} \sum_{\chi \in \mathcal{F}(q)} \frac{\tau(\chi)\mu(h)\chi(h)\bar{\chi}(m)}{m^{\frac{1}{2}-\alpha}h^{\frac{1}{2}+\gamma}q^{\frac{1}{2} + \alpha}}\nonumber\\
&=& \frac{G_-(\alpha)}{2} \sum_{h=1}^\infty \sum_{m\leq y}\sum_{\chi} \sum_{a=1}^{q-1} \frac{\chi(ah)\bar{\chi}(m) \mu(h)e\left(\frac{a}{q}\right)}{q^{\frac{1}{2}+\alpha}m^{\frac{1}{2}-\alpha}h^{\frac{1}{2}+\gamma}} \nonumber\\
&=& \frac{q-1}{2 q^{\frac{1}{2}+\alpha}} G_-(\alpha) \sum_{ah=m (q)} \frac{e\left(\frac{a}{q}\right)\mu(h)}{m^{\frac{1}{2}-\alpha}h^{\frac{1}{2}+\gamma}} - \frac{G_-(\alpha)}{2 q^{\frac{1}{2}+\alpha}} \sum_{h=1}^\infty \sum_{m\leq y}\sum_{a=1}^{q-1}\frac{e\left(\frac{a}{q}\right)\mu(h)}{m^{\frac{1}{2}-\alpha}h^{\frac{1}{2}+\gamma}} \nonumber\\
&=& K_1 + K_2.\eea

We analyze $K_2$ first. As always, we ignore all convergence issues in replacing a sum with an Euler product. The sum over $a$ gives -1 (if we had a sum over all $a$ modulo $q$ the exponential sum would vanish). As in the proof of Lemma \ref{lem:ratiosconjmathcals1}, the $m$-sum gives $\zeta(1/2-\alpha)$ and the $h$-sum gives $1/\zeta(1/2+\gamma)$. Thus \be K_2 \ = \ - \frac{G_-(\alpha)}{2 q^{\frac{1}{2}+\alpha}} \frac{\zeta\left(\foh - \alpha\right)}{\zeta\left(\foh+\gamma\right)}. \ee

In analyzing $K_1$, we find ourselves in a similar situation as the one we encountered in Lemma \ref{lem:ratiosconjmathcals1}. There is only a contribution when $ah \equiv m \bmod q$, in which case we find \be K_1 \ = \  \frac{q-1}{2 q^{\frac{1}{2}+\alpha}} G_-(\alpha) \sum_{ah=m (q)} \frac{e\left(\frac{a}{q}\right)\mu(h)}{m^{\frac{1}{2}-\alpha}h^{\frac{1}{2}+\gamma}}. \ee For similar reasons, we expect this piece to be small. We have enormous oscillation in the numerator, we have a congruence $ah \equiv m \bmod q$ which drastically reduces the number of summands, and the piece is multiplied by a factor of the order $q^{1/2 - \alpha}$, which when divided by the family's cardinality and differentiated will give a piece on the order of $q^{-1/2+\epsilon}$. We thus don't expect a contribution to the derivative of $R_{\mathcal{F}(q)}(\alpha,\gamma)$ from this piece; see Remark \ref{rek:howhandlek1} for additional comments.

For $\mathcal{S}_2$, having $\chi(-h)$ instead of $\chi(h)$ now leads to $a=-1$ and $h=m=1$ for the main term, giving \be \mathcal{S}_2 \ = \ \frac{(q-1) G_+(\alpha) e(-1/q)}{2q^{1/2+\alpha}} - \frac{G_+(\alpha)}{2q^{1/2+\alpha}} \frac{\zeta\left(\foh-\alpha\right)}{\zeta\left(\foh+\gamma\right)}. \ee
\end{proof}

\begin{rek}\label{rek:howhandlek1} In the analysis of $K_1$ above, we only kept the $a=m=h=1$ term. If we do want to attempt to analyze this term's contributions in greater detail, arguing in a similar manner as in Lemma \ref{lem:ratiosconjmathcals1} and Remark \ref{rek:droppingtermsinmathcalS1} gives that the `main' component of this sum is probably from $a=m=h=1$, which gives $\frac{q-1}{2q^{1/2+\alpha}}\ G_-(\alpha) e(1/q)$. As $\mathcal{S}_2$ and $\mathcal{S}_3$ are not present if we assume the standard form of the Ratios Conjecture, we content ourselves with just keeping this `main' term here (namely the $a=m=h=1$ term); a more detailed analysis keeping more terms would lead to the same final result in the 1-level density. Alternatively, we could look at $K_1$ as being multiplied by a factor of $q-1$ relative to $K_2$; however, the congruence means we have on the order of $1/q$ as many terms, and thus $K_1$ and $K_2$ should lead to similarly sized contributions to the 1-level density. As $K_2$ contributes $O(q^{-1/2+\gep})$ to the 1-level density, it is reasonable to posit a similarly sized contribution for $K_1$. Further, unlike the analysis of $\mathcal{S}_1$, \emph{every} term of $K_1$ has some $\alpha$ dependence through the factor of $q^{1/2+\alpha}$ in the denominator. Thus there will always be oscillation with $\alpha$, and by the `diagonal' comments in \cite{CFZ1} and \cite{CS} this term should yield a negligible contribution.
\end{rek}

\subsection{Differentiation and the contour integral}\label{sec:diffandcontourint} We follow \cite{CS,Mil5} to determine the Ratios Conjecture's prediction for the $1$-level density. The first step is to compute the derivative of $R_{\mathcal{F}(q)}(\alpha,\gamma)$.

\begin{lem}\label{lem:rationsconjpredderiv} Let $G_\pm(\alpha)$ be as in \eqref{eq:defnGpm}.
We have \bea & & \frac{\partial R_{\mathcal{F}(q)}}{\partial \alpha}\Bigg|_{\alpha=\gamma=r}
\nonumber\\ & &  = \ \frac{q-1}{2q^{1/2}}\Bigg[\frac{G_+'(r) q^r - r G_+(r)q^{r-1}}{q^{2r}}\ e(-1/q)+ \frac{G_-'(r) q^r - r G_-(r)q^{r-1}}{q^{2r}}\ e(1/q)\Bigg] \nonumber\\ & & \ \ - \  \frac1{2q\zeta\left(\frac12+r\right)} \Bigg[ \frac{\left(\left(G_+'(r)+G_-'(r)\right)\zeta\left(\frac12-r\right) + \left(G_+(r)+G_-(r)\right)\zeta'\left(\frac12-r\right)\right)q^r}{q^{2r}} \nonumber\\ & & \ \ \ \ \ \ \ \ \ -\ \frac{r \left(G_+(r)+G_-(r)\right)\zeta\left(\frac12-r\right)q^{r-1}}{q^{2r}}    \Bigg]
\nonumber\\ & & \ \ - \ \frac{\zeta'\left(\frac12+r\right)}{\zeta\left(\frac12+r\right)} \ + \ O\left(q^{1/2+\gep}\right), \eea where the bracketed quantities are present or not depending on whether we are arguing as in the standard Ratios Conjecture or instead not automatically dropping any terms multiplied by signs of functional equations averaging to zero; as the contribution from these terms will be $O(q^{-1/2+\gep})$, it is immaterial whether or not we include them.
\end{lem}

\begin{proof} The proof follows from a straightforward differentiation of \eqref{eq:ratiospredRU}.
\end{proof}

\begin{rek} Note there is no $q$-dependence in $G_\pm(\alpha)$, and thus its derivatives are independent of $q$. \end{rek}

We now prove Theorem \ref{thm:mainratiospred}. Recall it was\\

\noindent \textbf{Theorem \ref{thm:mainratiospred}}. \emph{Denote the $1$-level density for $\mathcal{F}(q)$ (the family of non-principal Dirichlet characters modulo a prime $q$) by \be
D_{1,\mathcal{F}(q)}(\phi) \ := \ \frac1{q-2}\sum_{\chi\in \mathcal{F}(q)}
\sum_{\gamma_\chi \atop L(1/2+i\gamma_\chi,\chi) = 0} \phi\left(\gamma_\chi
\frac{\log \frac{q}{\pi}}{2\pi}\right), \ee with $\phi$ an even Schwartz function whose Fourier transform has compact support.  The Ratios Conjecture's prediction for the $1$-level density of the family of non-principal Dirichlet characters modulo $q$ is} \bea D_{1,\mathcal{F}(q)}(\phi) & \ = \ &  \hphi(0) + \frac1{\log \frac{q}{\pi}} \sum_{\chi\in\mathcal{F}(q)} \int_{-\infty}^\infty \phi(\tau) \left[ \frac{\Gamma'}{\Gamma}\left(\frac14 + \frac{a(\chi)}{2}+\frac{\pi i \tau}{\log \frac{q}{\pi}}\right) \right]d\tau \nonumber\\ & & \ \ \ \ + \ O\left(q^{-1/2+\gep}\right) .\eea

\begin{proof} As the argument is essentially the same as in \cite{CS,Mil5}, we merely highlight the proof. We first compute the unscaled $1$-level density with $g$ an even
Schwartz function:  \be S_{1;\mathcal{F}(q)}(g) \ = \ \frac1{q-2}\sum_{\chi\in \mathcal{F}(q)}
\sum_{\gamma_\chi \atop L(1/2+i\gamma_\chi,\chi) = 0} g\left(\gamma_\chi
\right). \ee Let $c \in \left(\foh+\frac1{\log q},
\frac34\right)$; thus \bea S_{1;\mathcal{F}(q)}(g) & \ = \ & \frac1{q-2}\sum_{\chi\in \mathcal{F}(q)} \frac1{2\pi i} \left(\int_{(c)} - \int_{(1-c)}\right)
\frac{L'(s,\chi)}{L(s,\chi)}
g\left(-i\left(s-\foh\right)\right)ds \nonumber\\ &=&
S_{1,c;\mathcal{F}(q)}(g) + S_{1,1-c;\mathcal{F}(q)}(g). \eea We argue as
in \S3 of \cite{CS} or \S3 of \cite{Mil5}. We first analyze the integral on the line
$\Re(s) = c$. By GRH and the rapid decay of $g$, for large $t$ the
integrand is small. We use the Ratios Conjecture (Lemma
\ref{lem:rationsconjpredderiv} with $r = c - \foh + it$) to replace the
$\sum_\chi L'(s,\chi)/L(s,\chi)$ term when $t$ is small. We may then extend the integral to all of $t$ because of the rapid decay of $g$. As the integrand is regular at $r=0$ we can move
the path of integration to $c=1/2$. The contribution from the integral on the $c$-line is now readily bounded, as $\partial R_{\mathcal{F}(q)}/\partial \alpha\Big|_{\alpha=\gamma=r}$ is just the contribution from $\zeta'(1/2+r)/\zeta(1/2+r) + O(q^{1/2+\gep})$. As we divide by $q-2$, the big-Oh term is negligible. Note that the $\zeta'/\zeta$ term is independent of $q$, and thus gives a contribution of size $O(1/q)$ when we divide by the family's cardinality.

We now study $S_{1,1-c;\mathcal{F}(q)}(g)$: \bea & &
S_{1,1-c;\mathcal{F}(q)}(g)\nonumber\\ & & \ = \ \frac1{q-2}\sum_{\chi\in \mathcal{F}(q)}
\frac{-1}{2\pi i}  \int_{\infty}^{-\infty}
\frac{L'(1-(c+it),\chi)}{L(1-(c+it),\chi)}
g\left(-i\left(\foh-c\right)-t\right) (-idt).\ \ \ \ \ \ \ \ \  \eea  We use \eqref{eq:negLprimeoverL}, a consequence of the functional equation, with $s=c+it$. We get another $\sum_\chi L'(c+it,\chi)/L(c+it,\chi)$, which does not contribute by Lemma \ref{lem:rationsconjpredderiv}. We again shift contours to $c=1/2$. We are left with the integral against $\log(q/\pi)$ and the two Gamma factors, which may be combined as $\phi$ is even when $c=1/2$. We are left with \bea S_{1,1-c;\mathcal{F}(q)}(g) & \ = \ & \frac1{2\pi} \int_{-\infty}^\infty \left[\log \frac{q}{\pi} + \frac{\Gamma'}{\Gamma}\left(\frac14 + \frac{a(\chi)}{2}+\frac{\pi i \tau}{\log \frac{q}{\pi}}\right) \right] g(t)dt + O\left(q^{-1/2+\gep}\right).\nonumber\\  \eea

%We use the
%functional equation \be L(s,\chi) \ = \ \epsilon_\chi X_L(s,\chi) L(1-s,\chi) \ee
%to find that \be \frac{L'(1-(c+it),\chi)}{L(1-(c+it),\chi)} \ = \ -
%\frac{L'(c+it,\chi)}{L(c+it,\chi)} + \frac{X_L'(c+it,\chi)}{X_L(c+it,\chi)}. \ee

In investigating zeros near the central point, it is convenient to
renormalize them by the logarithm of the analytic conductor. Let
$g(t) = \phi\left(\frac{t \log (q/\pi)}{2\pi}\right)$. A straightforward
computation shows that $\widehat{g}(\xi) = \frac{2\pi}{\log (q/\pi)}\
\hphi(2\pi \xi / \log \frac{q}{\pi})$. The (scaled) $1$-level density for the
family $\mathcal{F}(q)$ is therefore \be D_{1,\mathcal{F}(q);R}(\phi) \ = \  \frac1{q-2}
\sum_{\chi\in \mathcal{F}(q)} \sum_{\gamma_\chi \atop
L(1/2+i\gamma_\chi,\chi) = 0} \phi\left(\gamma_\chi \frac{\log
\frac{q}{\pi}}{2\pi}\right) \ = \ S_{1;\mathcal{F}(q)}(g) \ee (where $g(t) =
\phi\left(\frac{t \log R}{2\pi}\right)$ as before). We replace $g(t)$ with $\phi( t\log(q/\pi) / 2\pi)$, and then change variables by letting $\tau = t \log(q/\pi)/2\pi$ and we find \bea
D_{1,\mathcal{F}(q);R}(\phi)  & \ = \ &   \frac1{\log \frac{q}{\pi}}\sum_{\chi\in\mathcal{F}(q)} \int_{-\infty}^\infty \phi(\tau) \left[\log \frac{q}{\pi} +  \frac{\Gamma'}{\Gamma}\left(\frac14 + \frac{a(\chi)}{2}+\frac{\pi i \tau}{\log \frac{q}{\pi}}\right) \right]d\tau \nonumber\\ & & \ \ \ \ + \ O\left(q^{-1/2+\gep}\right) \nonumber\\ & = & \hphi(0) + \frac1{\log \frac{q}{\pi}} \sum_{\chi\in\mathcal{F}(q)} \int_{-\infty}^\infty \phi(\tau) \left[ \frac{\Gamma'}{\Gamma}\left(\frac14 + \frac{a(\chi)}{2}+\frac{\pi i \tau}{\log \frac{q}{\pi}}\right) \right]d\tau \nonumber\\ & & \ \ \ \ + \ O\left(q^{-1/2+\gep}\right) .\eea
\end{proof}

%%%%%%%%%%%%%%%%%%%%%%%%%%%%%%%%%%%%%%%%%%%%%%%%%%%%%%%%%%%%%%%%%%%%%%%%%%%%%%%%%%%%%%%%%%%%%%%%
%%%%%%%%%%%%%%%%%%%%%%%%%%%%%%%%%%%%%%%%%%%%%%%%%%%%%%%%%%%%%%%%%%%%%%%%%%%%%%%%%%%%%%%%%%%%%%%%
%%%%%%%%%%%%%%%%%%%%%%%%%%%%%%%%%%%%%%%%%%%%%%%%%%%%%%%%%%%%%%%%%%%%%%%%%%%%%%%%%%%%%%%%%%%%%%%%

\section{Theoretical Results}

We prove Theorem \ref{thm:mainNT}. The first step is the explicit formula for $\mathcal{F}(q)$, the family of non-principal, primitive characters to a prime modulus $q$ (remember there are $q-2$ such characters). The calculations below are similar to those in \cite{HR,Mil6}, the primary difference being that here we are interested in computing the error terms down to square-root cancelation, whereas in these papers the purpose was to compute just the main term.

Let $\phi$ be an even Schwartz function whose Fourier transform has compact support in $(-\sigma, \sigma)$. The explicit formula (see \cite{Mil3,RS}) gives the following for the 1-level density for the family:
\begin{align} \frac{1}{q-2}&\sum_{\chi \in \mathcal{F}(q)} \sum_\gamma \phi\left(\gamma \frac{\log \frac{q}{\pi}}{2\pi}\right) \nonumber \\
&= \frac{1}{(q-2)\log\frac{q}{\pi}} \int_{-\infty}^\infty \phi(\tau)\sum_{\chi \in \mathcal{F}(q)} \left[\log\frac{q}{\pi}+\frac{\Gamma'}{\Gamma}\left(\frac{1}{4}+\frac{a(\chi)}{2}+\frac{\pi i \tau}{\log \frac{q}{\pi}}\right)\right]\,d\tau  \nonumber \\
&-\frac{2}{(q-2)\log \frac{q}{\pi}}\sum_{\chi \in \mathcal{F}(q)} \sum_{k=1}^\infty \sum_{p}\frac{\chi(p)^k\log p}{p^{k/2}}\ \widehat{\phi} \left( \frac{\log p^k}{\log \frac{q}{\pi}} \right),
\end{align} where \be \twocase{a(\chi) \ := \ }{0}{if $\chi(-1) = 1$}{1}{if $\chi(-1) =
-1$.}\ee
This simplifies to
\begin{align} \frac{1}{(q-2)} &\sum_{\chi \in \mathcal{F}(q)} \sum_\gamma \phi\left(\gamma \frac{\log \frac{q}{\pi}}{2\pi}\right) \nonumber \\
     &=\ \widehat{\phi}(0) + \frac{1}{(q-2)\log \frac{q}{\pi}} \int_{-\infty}^\infty \phi(\tau)\sum_{\chi \in \mathcal{F}(q)} \left[\frac{\Gamma'}{\Gamma}\left(\frac{1}{4}+\frac{a(\chi)}{2}+\frac{\pi i \tau}{\log \frac{q}{\pi}}\right)\right]\,d\tau \nonumber \\
     &\ \ \ -\  \frac{2}{(q-2)\log \frac{q}{\pi}}\sum_p \sum_{k=1}^\infty \sum_{\chi \in \mathcal{F}(q)} \frac{\chi(p)^k\log p}{p^{k/2}}\ \widehat{\phi}\left(\frac{\log p^k}{\log \frac{q}{\pi}}\right).
\end{align}

As the integral against the $\Gamma'/\Gamma$ piece directly matches with the prediction from the Ratios Conjecture, to prove Theorem \ref{thm:mainNT} it suffices to study the triple sum piece. We do this in the following lemma.

\begin{lem} Let $\supp(\hphi) \subset (-\sigma, \sigma) \subset (-2,2)$. For any $\gep > 0$ we have
\be \frac{1}{(q-2)\log \frac{q}{\pi}}\sum_p \sum_{k=1}^\infty \sum_{\chi \in \mathcal{F}(q)} \frac{\chi(p)^k\log p}{p^{k/2}}\ \widehat{\phi}\left(\frac{\log p^k}{\log \frac{q}{\pi}}\right) \ = \ O(q^{\frac{\sigma}{2}-1+\epsilon}). \ee In particular, these terms do not contribute for $\sigma < 2$, and contribute at most at the level of square-root cancelation for $\sigma < 1$.
\end{lem}

\begin{proof} Let \be \twocase{\delta_{1;k}(p) \ := \ }{1}{if $p^k \equiv 1 \bmod q$}{0}{otherwise.} \ee By the orthogonality relations for Dirichlet characters (Lemma \ref{lem:sumsdirichletcharinFq}), we have\be \sum_{\chi \in \mathcal{F}(q)} \chi(p)^k \ = \ -1 + \sum_{\chi \bmod q} \chi(p^k) \ = \ -1 + (q-1)\delta_{1;k}(p). \ee
Thus \begin{align}
 S &\ :=\ \frac1{(q-2)\log q}  \sum_{p}  \sum_{k = 1}^{\infty}  \sum_{\chi \in \mathcal{F}(q)}  \frac{\chi (p)^k \log p}{p^{k/2}} \widehat{\phi} \left(\frac{\log p^k}{\log \frac{q}{\pi}}\right) \nonumber \\
  &\ =\  \frac1{(q-2)\log q} \sum_{p}  \sum_{k = 1}^{\infty} \frac{\log p}{p^{k/2}} \widehat{\phi} \left(\frac{\log p^k}{\log \frac{q}{\pi}}\right) \sum_{\chi \in \mathcal{F}(q)} \chi (p)^k  \nonumber \\
  &\ \ll\ \frac1{q\log q} \sum_{p=2}^{q^{\sigma}} \sum_{k = 1}^{\frac{\sigma \log (q/\pi)}{\log p}} \frac{\log p}{p^{k/2}} \left|\sum_{\chi (q)} \chi (p)^k\right| \nonumber \\
  &\ \ll\ \frac1{q}  \sum_{p=2}^{q^{\sigma}} \sum_{k = 1}^{\frac{\sigma \log (q/\pi)}{\log p}} \frac{1}{p^{k/2}} \left|\sum_{\chi (q)} \chi (p)^k\right| \nonumber \\
  &\ =\  \frac1{q}  \sum_{p=2}^{q^{\sigma}} \sum_{k = 1}^{\frac{\sigma \log (q/\pi)}{\log p}} \frac{1}{p^{k/2}}\left|-1 + (q-1)\delta_{1;k}(p)\right| \nonumber \\
  &\ \ll\ \frac1{q}\left( \sum_{p=2}^{q^{\sigma}} \sum_{k = 1}^{\frac{\sigma \log (q/\pi)}{\log 2}} \frac{1}{p^{k/2}} + q \sum_{k = 1}^{\frac{\sigma \log (q/\pi)}{\log 2}} \sum_{p=2 \atop p^k \equiv 1 \bmod q}^{q^\sigma} \frac{1}{p^{k/2}} \right)\nonumber\\ &\ :=\ S_1 + S_2,
\end{align} where in the above sums we increased their values by increasing the upper bounds of the $k$-sums.

We bound $S_1$ first. We have
\bea
S_1 & \ \leq\ & \frac1{q}\sum_{p=2}^{q^{\sigma}} \sum_{k = 1}^{\infty} \frac{1}{p^{k/2}}\nonumber\\ & = & \frac1{q}\sum_{p=2}^{q^{\sigma}} \frac{p^{-1/2}}{1-p^{-1/2}} \nonumber \\
         &\leq & \frac1{q}\sum_{n=2}^{q^{\sigma}} \frac{1}{n^{1/2}-1} \nonumber\\
         &\ll &\frac1{q}\int_{2}^{q^{\sigma}} \frac{1}{x^{1/2}} \, dx  \ \ll \ \frac1{q} \cdot q^{\frac{\sigma}{2}}.\eea
Thus $S_1 \ll q^{\frac{\sigma}{2}-1}$, which is negligible for $\supp(\widehat{\phi}) \subset (-2,2)$, and gives an error of size one over the square-root of the family's cardinality for support up to $(-1,1)$.

As we made numerous approximations above, it is worth noting that $S_1$ will be at least of this size due to the contribution from the $k=1$ piece. To obtain better results would require us to exploit oscillation, which we cannot do as we are taking the absolute value of the character sums.

The analysis of $S_2$ depends crucially on when $p^k \equiv 1 \bmod q$. We find

%As $(\Z/q\Z)^\ast$ is a cyclic group of order $q-1$, for $q$ an odd prime we note that $2|q-1$, and thus there will be many primes

\bea\label{eq:expansionS2sum} S_2 & \ := \ &  \sum_{k = 1}^{\frac{\sigma \log (q/\pi)}{\log 2}} \sum_{p=2 \atop p^k \equiv 1 \bmod q}^{q^\sigma} \frac{1}{p^{k/2}} \nonumber\\
         & \ll &  \sum_{p=2 \atop p \equiv 1 \bmod q}^{q^\sigma} \frac{1}{p^{1/2}}
          + \sum_{p=2 \atop p^2 \equiv 1 \bmod q}^{q^\sigma} \frac{1}{p}
          + \sum_{k = 3}^{\frac{\sigma \log (q/\pi)}{\log 2}} \sum_{p=2}^{q^{\sigma}} \frac{1}{p^{k/2}} \nonumber\\ & := & B_1 + B_2 + B_3.  \eea

For the first sum above, note that since $p$ is a prime congruent to 1 modulo $q$, we may write $p = \ell q + 1$ for $\ell \ge 1$, and $\supp(\hphi) \subset (-\sigma, \sigma)$ restricts us to $\ell \le q^{\sigma-1}$. Thus the first sum in \eqref{eq:expansionS2sum} is bounded by \bea B_1 & \ \ll \ & \sum_{\ell=1}^{q^{\sigma-1}} \frac{1}{(q\ell + 1)^{1/2}} \nonumber\\
     &\leq & \frac{1}{q^{1/2}} \sum_{\ell = 1}^{q^{\sigma - 1}} \frac{1}{\ell^{1/2}} \nonumber \\
    &\ll &  \frac{1}{q^{1/2}} \cdot q^{\frac{\sigma - 1}2} \ \ll \ q^{\frac{\sigma}{2} -1}.  \eea

The second sum, $B_2$, is handled similarly. As $p^2 \equiv 1 \bmod q$ and $\hphi$ is supported in $(-\sigma, \sigma)$, this means either $p = \ell q - 1$ or $\ell q + 1$ for $\ell \ge 1$. We find \bea B_2 & \ \ll \ & \sum_{\ell=1}^{q^{\sigma-1}} \frac1{\ell q} \ \ll \ \frac{\log q}{q}; \eea note this term is negligible for any finite support.

The proof is completed by bounding $B_3$. Note for each $k$, $p^k \equiv 1 \bmod q$ is the union of at most $k$ arithmetic progressions (with $k \ll \log q$), and the smallest $p$ can be is $q^{1/k}$ (as anything smaller has its $k$\textsuperscript{th} power less than $q$). The actual smallest $p$ can be significantly larger, as happened in the $k=2$ case where the smallest $p$ could be is $q-1$, much larger than $\sqrt{q}$.  We may replace the prime sum with a sum over $\ell \ge 0$ of $1/(\ell q + q^{1/k})^{k/2}$ (as in the previous cases, it is a finite sum due to the compact support of $\hphi$). Thus \bea
B_3 & \ \ll \ &  \sum_{k = 3}^{\frac{\sigma \log (q/\pi)}{\log 2}} \sum_{\ell=0}^{q^{\sigma-1}}  \frac{\log q}{(\ell q + q^{1/k})^{k/2}} \nonumber\\ & \ll & \sum_{k=3}^{2\sigma \log q} \left[\frac{\log q}{q^{1/2}} +
  \sum_{\ell=1}^{q^{\sigma-1}} \frac{1}{(\ell q + q^{1/k})^{k/2}}\right]  \nonumber \\ & \ll & \frac{\log^2 q}{q^{1/2}} + \frac{\log^2 q}{q^{3/2}} \sum_{\ell=1}^{q^{\sigma-1}} \frac1{\ell^{3/2}} \ \ll \ \frac{\log^2 q}{q^{1/2}},\eea which again is negligible for all support.\end{proof}

\begin{rek} We can improve the error term arising from $B_3$ beyond square-root cancelation by assuming more about $q$. For example, if $q$ and $(q-1)/2$ are primes ($(q-1)/2$ is called a Sophie Germain prime), then there are no primes $p$ with $p^k \equiv 1 \bmod q$ that contribute with our support restrictions for $k \ge 3$. We do not pursue such an analysis here for two reasons: (1) we don't expect to be able to get errors better than square-root cancelation elsewhere; (2) while standard conjectures imply the infinitude of Germain primes, there are no unconditional proofs of the existence of infinitely many such $q$, though the Circle Method predicts there should be about $2C_2 x / \log^2 x$ Sophie Germain primes at most $x$, where $C_2 \approx .66016$ is the twin prime constant; see \cite{MT-B} for the calculation.\end{rek}

\begin{rek} As mentioned in the introduction, assuming Conjectures \ref{conj:weakmont} or \ref{conj:miller} allow us to extend the support in the number theory computations beyond $(-2,2)$. This is done in \cite{Mil6}, and the results agree with the Ratios Conjecture's prediction.
\end{rek}

%%%%%%%%%%%%%%%%%%%%%%%%%%%%%%%%%%%%%%%%%%%%%%%%%%%%%%%%%%%%%%%%%%%%%%%%%%%%%%%%%%%%%%%%%%%%%
%%%%%%%%%%%%%%%%%%%%%%%%%%%%%%%%%%%%%%%%%%%%%%%%%%%%%%%%%%%%%%%%%%%%%%%%%%%%%%%%%%%%%%%%%%%%%

\ \\

\end{document}